\documentclass[aap]{imsart}

\RequirePackage{amsthm,amsmath,amsfonts,amssymb}
\RequirePackage[numbers]{natbib}
\RequirePackage[colorlinks,citecolor=blue,urlcolor=blue]{hyperref}
\RequirePackage{graphicx}
\RequirePackage{bbm}
\RequirePackage{enumitem}

\startlocaldefs
\theoremstyle{plain}
\newtheorem{theorem}{Theorem}[section]
\newtheorem{proposition}[theorem]{Proposition}
\newtheorem{lemma}[theorem]{Lemma}
\newtheorem{corollary}[theorem]{Corollary}

\theoremstyle{definition}
\newtheorem{definition}[theorem]{Definition}
\newtheorem*{remark}{Remark}
\newtheorem{example}[theorem]{Example}
\def\ci{\perp\!\!\!\perp}
\endlocaldefs

\begin{document}

\begin{frontmatter}
\title{Calibrated Probability Forecast Sequences and Measure-Valued Martingales}
\runtitle{Calibrated Probability Forecast Sequences}

\begin{aug}
\author[A]{\fnms{Thomas S.}~\snm{Wilkinson}\ead[label=e1]{tw636@exeter.ac.uk}\orcid{0009-0006-0654-6145}}
\and
\author[A]{\fnms{Christopher A. T.}~\snm{Ferro}\ead[label=e2]{c.a.t.ferro@exeter.ac.uk}\orcid{0000-0002-9830-9270}}
\address[A]{Department of Mathematics and Statistics, University of Exeter\printead[presep={,\ }]{e1,e2}}
\end{aug}

\begin{abstract}
We consider the calibration of probability forecasts. Several notions of calibration exist when the forecaster issues a single forecast for each of the observations that is to be predicted. We extend one of these notions, auto-calibration, to the common situation in which the forecaster issues a sequence of forecasts for each observation, repeatedly updating their prediction as they receive additional information. For observations that sit in any Borel space, we show that auto-calibration is equivalent to a certain sequence of random probability measures satisfying the martingale property, and we propose a simple, statistical approach to testing this property. This provides, for the first time, a way of testing the calibration of such sequences of probability forecasts.
\end{abstract}



\end{frontmatter}

For the purpose of open access, the authors have applied a Creative Commons Attribution (CC BY) licence to any Author Accepted Manuscript version arising from this submission.

\section{Introduction}
A probability forecast takes the form of a probability measure on the set of possible values of the observation that is to be predicted. We shall denote a typical probability forecast by $\mu$, where $\mu$ is a random probability measure. This allows the forecast to be random---in which case we implicitly refer to what the forecaster would have produced in a counterfactual reality, and how likely it was for that counterfactual to have occured---but also allows the forecast to be constant, simply equal to what the forecaster actually produced.

We shall consider a property of probability forecasts called \textit{calibration}. If forecasts are calibrated then the probabilities that they assign to events match the probabilities with which those events occur in observations (see \cite{Gnt07}). In fact, there are several definitions of calibration, which vary in the strictness of the required consistency between the forecasts and observations. For example, a probability forecast, $\mu$, for an observation, $Y$, is called \textit{auto-calibrated} if $\mu = \mathcal{L}(Y | \mu)$, and it is called \textit{ideal} with respect to a sigma algebra, $\Psi$, if $\mu = \mathcal{L}(Y | \Psi)$ (see for example \cite{Tsp13} and \cite{Gnt13}).

Assessing empirically whether or not forecasts are calibrated is an important step in improving forecasting systems as it can reveal biases in the forecasts, which forecasters can seek to correct. We cannot assess whether an individual forecast is calibrated, as we could only compare that forecast to a single observation value, but we can assess a collection of forecasts as a whole. A common approach is to consider \textit{rolling-event} forecast sequences, $(\mu_1, \ldots, \mu_n)$, in which each forecast, $\mu_i$, is made for a corresponding observation, $Y_i$, and the value of $Y_i$ is known to the forecaster when they make their next forecast, $\mu_{i+1}$. Statistical properties of the forecasts are then derived under the assumption that they are calibrated and these properties are tested empirically. For example, suppose that the observations are real-valued and that the distribution function, $F_i$, associated with $\mu_i$ is continuous. Then, if $\mu_i$ is auto-calibrated for each $i$, the probability integral transform (PIT) values, $F_i(Y_i)$, are uniformly distributed on $(0, 1)$ (\cite{Gnt13}). Furthermore, if, for each $i$, $\mu_i$ is ideal with respect to the $\sigma$-algebra generated by $(Y_1, \ldots, Y_{i-1})$ then the PIT values are also independent (\cite{Dwd84}, \cite{Dbl98}). Testing these properties of PIT values is a standard way of assessing the calibration of probability forecasts.

In contrast to rolling-event sequences, there has been little discussion of calibration for \textit{fixed-event} sequences of probability forecasts, in which every forecast is made for the same observation, with the forecaster repeatedly updating their prediction as they receive more information. Such sequences of forecasts are common in many fields and potentially contain a lot of information about the quality of the forecasting system.

In \cite{nrd87}, Nordhaus considers fixed-event sequences of \textit{point} forecasts, each of which consists of a single value that the observation could take. Nordhaus defines a calibration property for the fixed-event point forecast sequence $(Q_1, \ldots, Q_n)$ called `weak efficiency', which essentially says that each forecast is unbiased and incorporates the information of all previous forecasts in the sequence. Nordhaus also defines the revision from one point forecast to the next, which is simply the latter subtract the former. He shows that if the forecasts are weakly efficient then the forecast sequence is a martingale: $\mathbb{E}[Q_{i+1}|Q_1, \ldots, Q_i] = Q_i$, or equivalently $\mathbb{E}[Q_{i+1} - Q_i|Q_1, \ldots, Q_i] = 0$. Thus we can assess whether the forecast sequence is weakly efficient by assessing whether the forecast revisions have expectation $0$ and are uncorrelated.

In \cite{mtc08}, Mitchell considers whether the same approach can be taken to testing fixed-event sequences of probability forecasts. He shows that revisions between probability forecasts as given by the Kullback--Leibler Information Criterion need not have conditional expectation $0$. This is not surprising: the KLIC is non-negative, so in order for it to have expectation $0$ it must equal $0$ almost surely, and so the current forecast must equal the previous one almost surely. Mitchell also suggests fixing an event defined in terms of the observation, and extracting from each probability forecast the probability assigned to that event; calibration of the probability forecast sequence then implies that the sequence of extracted probabilities forms a martingale, which we can assess by calculating the revisions as in Nordhaus' method.

In this paper, we develop a treatment of the calibration of fixed-event sequences of probability forecasts that is analogous to the standard treatment for rolling-event sequences, which we outlined earlier. In particular, we define a property for fixed-event sequences which we call `auto-calibration' and develop testable consequences of it in terms of `synthetic' PIT values. In a pleasing parallel with fixed-event point forecast sequences, we find that a calibrated fixed-event probability forecast sequence is a measure-valued martingale, and our results about synthetic PIT values apply to any sequence of random probability measures satisfying the martingale condition.

In section \ref{fixed_event}, we define auto-calibration of a fixed-event probability forecast sequence and show how it can be written in terms of the martingale property. In section \ref{rolling_event}, we review the results for PIT values of rolling-event sequences that we shall adapt for fixed-event sequences. In section \ref{syn_PITs}, we develop our main results about calibrated probability forecast sequences and measure-valued martingales. In section \ref{revisions}, we re-frame some of these results in terms of revisions to probability forecasts. We give examples in section \ref{examples} and conclude with a discussion in section \ref{discussion}. Throughout, let $(\Omega, \mathcal{F}, \mathbb{P})$ be a probability space.

\section{Fixed-event probability forecast sequences}\label{fixed_event}
In this section we discuss a calibration property for fixed-event probability forecast sequences which we call auto-calibration.

Let $\mu$ be a probability forecast made for an observation $Y$. As mathematical objects this means $Y$ is a random element in a Borel space $(S, \mathcal{S})$ and $\mu$ is a random probability measure on $(S, \mathcal{S})$. We require $(S, \mathcal{S})$ to be Borel (see page 14 of \cite{Kll21}) to ensure the regular conditional distribution $\mathcal{L}(Y|\Psi)$ exists and is essentially unique for any $\sigma$-algebra $\Psi \subset \mathcal{F}$, using Theorem 8.5 of \cite{Kll21}. We use the standard $\sigma$-algebra on the set of probability measures on $(S, \mathcal{S})$, which is generated by the evaluation maps $(\mu \mapsto \mu(A))$ for $A \in \mathcal{S}$. If $(S, \mathcal{S}) = (\mathbb{R}, \mathcal{B}(\mathbb{R}))$ then this standard $\sigma$-algebra is in fact generated by the maps $(\mu \mapsto \mu((-\infty, y]))$ for $y \in \mathbb{R}$, by a monotone-class argument (see Theorem 1.1 of \cite{Kll21}).

For $\Psi \subset \mathcal{F}$ a $\sigma$-algebra, we say $\mu$ is ideally calibrated given $\Psi$ if $\mu = \mathcal{L}(Y|\Psi)$ (see e.g. Definition 1 of \cite{Tsp20}). We say $\mu$ is auto-calibrated if $\mu = \mathcal{L}(Y|\mu)$ (see e.g. page 5 of \cite{Tsp13}). If $\mu$ is ideally calibrated given $\Psi$ then it is auto-calibrated, by the Tower Law of conditional expectation. Also, if $\mu$ is auto-calibrated then it is ideally calibrated given $\sigma(\mu)$. Therefore, $\mu$ is auto-calibrated if and only if there exists a $\sigma$-algebra $\Psi \subset \mathcal{F}$ such that $\mu$ is ideally calibrated given $\Psi$.

We extend these definitions to fixed-event probability forecast sequences as follows.
\begin{definition}
  Let $\mu_1, \ldots, \mu_n$ be probability forecasts all made for an observation $Y$. Let $\Psi_1 \subset \ldots \subset \Psi_n \subset \mathcal{F}$ be a filtration. Then the sequence $(\mu_1, \ldots, \mu_n)$ is ideally calibrated given $(\Psi_1, \ldots, \Psi_n)$ if, for each $i \in \{1, \ldots, n\}$,
\begin{equation*}
\mu_i = \mathcal{L}(Y|\Psi_i).
\end{equation*}
\end{definition}

\begin{definition}\label{fixed_auto-calibration}
  Let $\mu_1, \ldots, \mu_n$ be probability forecasts all made for an observation $Y$. Then the sequence $(\mu_1, \ldots, \mu_n)$ is auto-calibrated if, for each $i \in \{1, \ldots, n\}$,
\begin{equation*}
\mu_i = \mathcal{L}(Y|\mu_1, \ldots, \mu_i).
\end{equation*}
\end{definition}

An ideally calibrated sequence of probability forecasts for $Y$ has each forecast in the sequence ideally calibrated for $Y$, but additionally the information sets $\Psi_1 \subset \ldots \subset \Psi_n$ used by the forecaster must be nested; and an auto-calibrated sequence of probability forecasts for $Y$ has each forecast in the sequence auto-calibrated for $Y$, but additionally the information set $\sigma(\mu_i)$ used by the forecaster to make forecast $\mu_i$ must include all of the prior forecasts in the sequence. It seems reasonable to impose these additional requirements on the forecast sequence, since the forecaster will usually continue to have access to the data they used to make their earlier forecasts, and should keep using those data when producing their future forecasts (as long as they are still relevant).

In Proposition \ref{auto-calibration_equivalence} we will give three properties of a fixed-event probability forecast sequence which are equivalent to auto-calibration, two of which are stated in terms of the martingale property of a sequence of random probability measures. First, we give our definition of the martingale property, which is a restatement of the definition of a `martingale measure' on page 219 of \cite{Hrw85}. It is also a special case of the very general Definition 3.1.1(iii) of \cite{Htn16}, in which the conditional expectation has been defined in terms of the Bochner integral. For our purposes, the conditional expectation $\mathbb{E}[\mu|\Psi]$ of a random probability measure $\mu$ on a Borel space $(S,\mathcal{S})$ given a $\sigma$-algebra $\Psi \subset \mathcal{F}$ is defined as in Lemma 2.10(iii) of \cite{Kll17} (where we can take the ring $\hat{\mathcal{S}}$ of bounded sets to be $\mathcal{S}$, and the conditional expectation is clearly locally finite). This means $\mathbb{E}[\mu|\Psi]$ is the essentially unique random probability measure on $(S,\mathcal{S})$ satisfying $\mathbb{E}[\mu|\Psi](A) = \mathbb{E}[\mu(A)|\Psi]$ for all $A \in \mathcal{S}$. (To show that this property determines $\mathbb{E}[\mu|\Psi]$ up to almost sure equality we can use the same style of argument as for regular conditional distributions.)

\begin{definition}\label{martingale}
  Let $(S,\mathcal{S})$ be a Borel space, and let $\mu_1, \ldots, \mu_n$ be random probability measures on $(S,\mathcal{S})$. Then the sequence $(\mu_1, \ldots, \mu_n)$ is a martingale (on $(S,\mathcal{S})$) if, for each $i \in \{1,\ldots,n-1\}$,
  \begin{equation*}
    \mathbb{E}[\mu_{i+1}|\mu_1,\ldots,\mu_i] = \mu_i.
  \end{equation*}
\end{definition}

For $s \in S$ we write $\delta_s$ for the corresponding Dirac measure on $(S,\mathcal{S})$, which is given by $\delta_s(A) = \mathbbm{1}(s \in A)$ for all $A \in \mathcal{S}$.

\begin{proposition}\label{auto-calibration_equivalence}
  Let $\mu_1, \ldots, \mu_n$ be probability forecasts made for an observation $Y$. Then the following are equivalent:
  \begin{enumerate}[label=(\roman*)]
  \item there exists a filtration $\Psi_1 \subset \ldots \subset \Psi_n \subset \mathcal{F}$ such that $(\mu_1, \ldots, \mu_n)$ is ideally calibrated given $(\Psi_1, \ldots, \Psi_n)$;
  \item the sequence $(\mu_1, \ldots, \mu_n)$ is auto-calibrated;
  \item the sequence $(\mu_1, \ldots, \mu_n)$ is a martingale and $\mu_n = \mathcal{L}(Y|\mu_1, \ldots, \mu_n)$;
  \item the sequence $(\mu_1, \ldots, \mu_n, \delta_Y)$ is a martingale.
\end{enumerate}
\end{proposition}

The proof of Proposition \ref{auto-calibration_equivalence} is straightforward, and is given in the \hyperref[appn]{Appendix}.

If $\mathcal{A}_1 \subset \ldots \subset \mathcal{A}_n \subset \mathcal{F}$ is a filtration where $\mathcal{A}_i$ represents all information available to the forecaster when they produced forecast $\mu_i$, then $(\mu_1, \ldots, \mu_n)$ being ideally calibrated given $(\mathcal{A}_1, \ldots, \mathcal{A}_n)$ means that every forecast in the sequence was the best probability forecast they could have produced when it was made. That ideally calibrated forecast sequence is then auto-calibrated, which indicates that auto-calibration is a desirable property for fixed-event probability forecast sequences. Auto-calibration in fact allows the forecaster to use only a subset $\Psi_i \subset \mathcal{A}_i$ of the information available to them and to produce the ideal forecast $\mathcal{L}(Y|\Psi_i)$, as long as $\Psi_1 \subset \ldots \subset \Psi_n$, which roughly means that when each forecast is made it has to use all of the information which was used to make the prior forecasts, at a minimum.

\begin{remark}
Tsyplakov defines conditional auto-calibration in Definition 2 of \cite{Tsp20} as follows: given a probability forecast $\mu$ made for observation $Y$, and a $\sigma$-algebra $\Delta \subset \mathcal{F}$, $\mu$ is conditionally auto-calibrated given $\Delta$ if $\mu = \mathcal{L}(Y|\Delta, \mu)$. This is equivalent to: there exists a $\sigma$-algebra $\Psi \subset \mathcal{F}$ such that $\Delta \subset \Psi$ and $\mu = \mathcal{L}(Y|\Psi)$. Using this definition we can say the fixed-event probability forecast sequence $(\mu_1, \ldots \mu_n)$ is auto-calibrated if and only if each forecast $\mu_i$ is conditionally auto-calibrated given $\sigma(\mu_1, \ldots, \mu_{i-1})$, the history of the sequence up to $\mu_i$.
\end{remark}

We now turn to the question of how to test whether a fixed-event probability forecast sequence is auto-calibrated. A possible starting point would be to extract from each forecast $\mu$ the functional value
\begin{equation*}
  \int h(y) \, \mu(dy),
\end{equation*}
for some function $h$, assuming this integral is well-defined; if the sequence of forecasts were a martingale then the resulting sequence of functional values would also be a martingale. However, we are not aware of a statistical test of the martingale property of a sequence of real numbers. In \cite{nrd87}, Nordhaus shows that if the sequence is a martingale then the `forecast revisions', each of which is one term in the sequence subtract the previous one, will all have expectation $0$ and will be uncorrelated; this can be assessed, but not tested without making further assumptions about the forecast revisions.

We will later present a method for testing auto-calibration which does not require any further assumptions about the forecasts. That method will be based on the Probability Integral Transform.

\section{Rolling-event probability forecast sequences}\label{rolling_event}
In this section, we give a technical review of the Probability Integral Transform (PIT). The PIT is used in a popular test of the calibration of rolling-event probability forecast sequences, and we prove the correctness of this test under a novel form of the null hypothesis. We will use this result in section \ref{syn_PITs} to present a test of the auto-calibration of fixed-event probability forecast sequences.

A function $F : \mathbb{R} \to \mathbb{R}$ is called a distribution function if it is non-decreasing, right-continuous and satisfies $\lim_{y \to -\infty} F(y) = 0$ and $\lim_{y \to \infty} F(y) = 1$; in other words, if it is the Cumulative Distribution Function of some $\mathbb{R}$-valued random variable. We use the standard $\sigma$-algebra on the set of distribution functions, which is generated by the evaluation maps $(F \mapsto F(y))$ for $y \in \mathbb{R}$. We write $\phi$ for the function which takes a probability measure $\mu$ on $\mathbb{R}$ to the distribution function $(y \mapsto \mu((-\infty, y]))$. The map $\phi$ is a bijection, and $\phi$ and $\phi^{-1}$ are both measurable.

For $Y$ an $\mathbb{R}$-valued random variable and $\Psi \subset \mathcal{F}$, we write $F_{Y|\Psi}$ for the regular conditional distribution of $Y$ given $\Psi$, in the form of a random distribution function; thus $F_{Y|\Psi}$ is the essentially unique random distribution function such that for all $y \in \mathbb{R}$, $F_{Y|\Psi}(y) = \mathbb{P}(Y \leq y|\Psi)$. For $F$ a random distribution function and $\Psi \subset \mathcal{F}$ a $\sigma$-algebra, we write $\mathbb{E}[F|\Psi]$ for the conditional expectation of $F$ given $\Psi$, meaning $\mathbb{E}[F|\Psi]$ is the essentially unique random distribution function such that for all $y \in \mathbb{R}$, $\mathbb{E}[F|\Psi](y) = \mathbb{E}[F(y)|\Psi]$. We can then extend Definitions \ref{fixed_auto-calibration} and \ref{martingale} to a sequence $(F_1, \ldots, F_n)$ of random distribution functions in place of a sequence $(\mu_1, \ldots, \mu_n)$ of random probability measures, in the obvious ways.

Given $F$ a distribution function and $y \in \mathbb{R}$, we write
\begin{equation*}
  F_-(y) = \lim_{\gamma \to y^-} F(\gamma),
\end{equation*}
where $\gamma \to y^-$ means the limit is taken as $\gamma$ approaches $y$ from below.

We begin by defining the following function, which we will use to define the PIT.
\begin{definition}
  The function $Z$ has arguments $F$ a distribution function, $y \in \mathbb{R}$ and $v \in [0,1]$ and takes values in $[0,1]$; it is given by
  \begin{equation*}
    Z(F, y, v) = (1-v) F_-(y) + v F(y).
  \end{equation*}
\end{definition}
The function $Z$ is measurable; to show measurability in $F$ for fixed $y$ and $v$, we can write
\begin{equation*}
  F_-(y) = \sup_{n \in \mathbb{Z}_+} F\left(y - \frac{1}{n}\right).
\end{equation*}

Note if $F$ is continuous then $Z(F, y, v) = F(y)$. When defining the PIT, and again in section \ref{revisions}, we will have a distribution function $F$ and an $\mathbb{R}$-valued random variable $Y$ and we will form the variable $Z(F, Y, V)$, where $V$ has a standard uniform distribution $\mathcal{U}([0,1])$ and is independent of $Y$. The following proposition then tells us the distribution of $Z(F, Y, V)$.

\begin{proposition}\label{distribution_of_Z}
  Let $F$ be a distribution function, and let $t \in (0,1)$. Let $y_t$ be given by
\begin{equation*}
y_t = \sup \{y \in \mathbb{R}|F(y) \leq t\},
\end{equation*}
and let $v_t$ be given by
\begin{equation*}
v_t = \begin{cases}
  \frac{t - F_-(y_t)}{F(y_t) - F_-(y_t)} & F_-(y_t) \neq F(y_t), \\
  1 & \text{otherwise.} \\
\end{cases}
\end{equation*}

  Then $0 \leq v_t \leq 1$ and $Z(F, y_t, v_t) = t$. Let $G$ be another distribution function. Then
\begin{equation*}
  \int_\mathbb{R} \int_0^1 \mathbbm{1}(Z(F, y, v) \leq t) \,dv \,dG(y) = Z(G, y_t, v_t).
\end{equation*}
\end{proposition}

We give a proof in the \hyperref[appn]{Appendix}.

In Lemma 1 of section 5.3 of \cite{Fgs67}, Ferguson states that if $Y$ is an $\mathbb{R}$-valued random variable with CDF $F$, and if $V$ is an independent random variable with distribution $\mathcal{U}([0,1])$, then $Z(F, Y, V) \sim \mathcal{U}([0,1])$. The forward implication of the following lemma is an equivalent statement and follows immediately from Proposition \ref{distribution_of_Z}.

\begin{lemma}\label{PIT_is_uniform}
  Let $F, G$ be distribution functions. Then $F = G$ if and only if for $t \in (0,1)$,
\begin{equation*}
  \int_\mathbb{R} \int_0^1 \mathbbm{1}(Z(F,y,v) \leq t) \,dv \,dG(y) = t.
\end{equation*}
\end{lemma}

For a proof of the backward implication see Lemma 3.2 of \cite{Mds23}. (In using Proposition \ref{distribution_of_Z} to prove the forward implication we have taken the approach outlined by Ferguson. Brockwell gives an alternative proof in Lemma 2.1 of \cite{Brc07}.)

The key insight underpinning the PIT is that $Z(F, Y, V)$ is uniform not only if $F$ is the CDF of $Y$, but also if $F$ is a probability forecast for $Y$ which is suitably calibrated. If $F$ is a forecast for $Y$, we define the PIT of $F$ to be $Z(F, Y, V)$, where $V$ is an arbitrary $\mathcal{U}([0,1])$ random variable independent of $Y$ (and independent of all other variables being considered). If $Z(F, Y, V) \sim \mathcal{U}([0,1])$ then $F$ is said to be probabilistically calibrated (see e.g. Definition 2.6(b) of \cite{Gnt13}).

The forward implication of the following proposition says that if $F$ is a forecast for $Y$ which is ideally calibrated with respect to $\Psi$, then $F$ is probabilistically calibrated, and in addition the PIT value is independent of $\Psi$.

\begin{proposition}\label{auto-calibration_PIT}
  Let $F$ be a distribution function, $\Psi \subset \mathcal{F}$ be a $\sigma$-algebra and $Y$ be an $\mathbb{R}$-valued random variable. Let $V$ be a random variable with distribution $\mathcal{U}([0,1])$ and independent of $\sigma(\Psi, Y)$.

  Then the following are equivalent:
\begin{enumerate}[label=(\roman*)]
  \item $F = F_{Y|\Psi}$;
  \item $F$ is $\Psi$-measurable, and $Z(F,Y,V)$ has distribution $\mathcal{U}([0,1])$ and is independent of $\Psi$.
\end{enumerate}
\end{proposition}
We give a proof in the \hyperref[appn]{Appendix}, based on the Disintegration Theorem (Theorem \ref{Disintegration}). (Gneiting and Ranjan give essentially the same proof of $\text{(i)} \Rightarrow \text{(ii)}$ in Theorem 2.8 of \cite{Gnt13}, although they do not show the independence of the PIT value from $\Psi$. Modeste gives a proof of the equivalence similar to ours in Proposition 3.4 of \cite{Mds23} but requires the underlying probability space $(\Omega, \mathcal{F}, \mathbb{P})$ to be Polish.)

We now consider a rolling-event sequence of probability forecasts $(F_1, \ldots, F_n)$ made for real-valued observations $(Y_1, \ldots, Y_n)$. Let $V_1, \ldots, V_n$ be independent and identitically distributed (abbreviated to i.i.d.) $\mathcal{U}([0,1])$ and independent of $\sigma(F_1, \ldots, F_n, Y_1, \ldots, Y_n)$; then for each $i \in \{1, \ldots, n\}$, the PIT value for $F_i$ is $Z_i = Z(F_i, Y_i, V_i)$. If $F_i$ is auto-calibrated then it is ideally calibrated with respect to $\sigma(F_i)$, so it is probabilistically calibrated by Proposition \ref{auto-calibration_PIT}. Thus if each individual forecast is auto-calibrated then the PIT values $Z_1, \ldots, Z_n$ all have distribution $\mathcal{U}([0,1])$, but they need not be independent. We will show in Proposition \ref{PIT_values} that the PIT values are i.i.d. $\mathcal{U}([0,1])$ under the following condition on the forecast sequence.
\begin{definition}\label{rolling_auto-calibration}
  Let $\mu_1$, \ldots, $\mu_n$ be random probability measures on a Borel space $(S,\mathcal{S})$ and let $Y_1$, \ldots, $Y_n$ be random elements in $(S,\mathcal{S})$. Then $(\mu_1, \ldots, \mu_n)$ is auto-calibrated as a rolling-event probability forecast sequence for observations $(Y_1, \ldots, Y_n)$ if, for each $i \in \{1,\ldots,n\}$,
  \begin{equation*}
    \mu_i = \mathcal{L}(Y_i|\mu_1, Y_1, \ldots, \mu_{i-1}, Y_{i-1}, \mu_i).
  \end{equation*}
\end{definition}
We extend Definition \ref{rolling_auto-calibration} to a sequence $(F_1, \ldots, F_n)$ of random distribution functions in place of $(\mu_1, \ldots, \mu_n)$ in the obvious way.

\begin{remark}
If a rolling-event forecast sequence $(\mu_1, \ldots, \mu_n)$ for observations $(Y_1, \ldots, Y_n)$ satisfies the condition in Definition \ref{rolling_auto-calibration}, we can just say it is auto-calibrated. The reason for defining the condition for general sequences of random probability measures is we will need it in cases where $\mu_1, \ldots, \mu_n$ are not forecasts, or are forecasts that were originally made for observations other than $Y_1, \ldots, Y_n$.
\end{remark}

\begin{remark}
In Definition 3.5 of \cite{Mds23}, Modeste gives the name `auto-calibration' to a different property of a rolling-event probability forecast sequence which does not imply the PIT values are i.i.d. $\mathcal{U}([0,1])$.
\end{remark}

\begin{remark}
Using Tsyplakov's definition of conditional auto-calibration again, we can say that a rolling-event sequence of probability forecasts $(\mu_1, \ldots, \mu_n)$ for observations $(Y_1, \ldots, Y_n)$ is auto-calibrated if and only if each $\mu_i$ is conditionally auto-calibrated given $\sigma(\mu_1$, $Y_1$, \ldots, $\mu_{i-1}$, $Y_{i-1})$, the history of the forecast and observation sequences up to $\mu_i$.
\end{remark}

\begin{proposition}\label{PIT_values}
  Let $F_1$, \ldots, $F_n$ be random distribution functions and $Y_1$, \ldots, $Y_n$ be $\mathbb{R}$-valued random variables. Let $V_1$, \ldots, $V_n$ be i.i.d. $\mathcal{U}([0,1])$ random variables independent of $\sigma(F_1, \ldots, F_n, Y_1, \ldots, Y_n)$.

  For each $i \in \{1, \ldots, n\}$, let $Z_i = Z(F_i, Y_i, V_i)$. Then the following are equivalent:
\begin{enumerate}[label=(\roman*)]
  \item $(F_1, \ldots, F_n)$ is auto-calibrated as a rolling-event probability forecast sequence for observations $(Y_1, \ldots, Y_n)$;
  \item for each $i \in \{1, \ldots, n\}$, $Z_i$ has distribution $\mathcal{U}([0,1])$ and is independent of 
\begin{equation*}
  \sigma(F_1, \ldots, F_i, Y_1, \ldots, Y_{i-1}).
\end{equation*}
\end{enumerate}

  Consequently, {\normalfont (i)} implies $Z_1$, \ldots, $Z_n$ are i.i.d. $\mathcal{U}([0,1])$.
\end{proposition}

\begin{proof}
  $\text{(i)} \Rightarrow \text{(ii)}$:\, For $i \in \{1, \ldots, n\}$, let $\mu_i = \phi^{-1}(F_i)$.

  Fix $i \in \{1, \ldots, n\}$. Since $(\mu_1, \ldots, \mu_n)$ is auto-calibrated as a rolling-event probability forecast sequence for observations $(Y_1, \ldots, Y_n)$, we have
\begin{equation*}
  \mu_i = \mathcal{L}(Y_i|\mu_1, Y_1, \ldots, \mu_{i-1}, Y_{i-1}, \mu_i).
\end{equation*}

  Then by Proposition \ref{auto-calibration_PIT}, $Z_i$ has distribution $\mathcal{U}([0,1])$ and is independent of
\begin{equation*}
  \sigma(\mu_1, Y_1, \ldots, \mu_{i-1}, Y_{i-1}, \mu_i),
\end{equation*}
  as required.

  $\text{(ii)} \Rightarrow \text{(i)}$:\, Fix $i \in \{1, \ldots, n\}$. $F_i$ is clearly $\sigma(F_1, Y_1, \ldots, F_{i-1}, Y_{i-1}, F_i)$-measurable, so by Proposition \ref{auto-calibration_PIT},
\begin{equation*}
  F_i = F_{Y_i|\sigma(F_1, Y_1, \ldots, F_{i-1}, Y_{i-1}, F_i)},
\end{equation*}
as required.

  $\text{Consequence of (i)}$:\, For $i \in \{1, \ldots, n\}$, by (ii), $Z_i$ has distribution $\mathcal{U}([0,1])$ and is independent of $\sigma(Z_1, \ldots, Z_{i-1})$.
\end{proof}

\begin{remark}
In Corollory 3.6 of \cite{Mds23}, Modeste shows that the PIT values are i.i.d. $\mathcal{U}([0,1])$ if each forecast $F_i$ is ideally calibrated with respect to the $\sigma$-algebra $\Psi_i$, where $\Psi_1 \subset \ldots \subset \Psi_n$ and $\sigma(Y_1, \ldots, Y_{i-1}) \subset \Psi_i$. Our definition of auto-calibration for a forecast sequence is equivalent to there existing $\sigma$-algebras $\Psi_1$, \ldots, $\Psi_n$ such that those conditions are satisfied.
\end{remark}

There are many documented approaches to testing that a sequence of numbers came from an i.i.d. $\mathcal{U}([0,1])$ sequence of random variables, including the Kolmogorov--Smirnov test, the Cram\'er--von Mises test and Neyman's smooth test. Any of these can be applied to a sequence of PIT values to test the auto-calibration of the forecast sequence.

We now summarise an approach described by Kn\"uppel, Kr\"uger and Pohle in \cite{Knp23}, using the PIT to test the calibration of a rolling-event forecast sequence where the forecasts are probability distributions on a space other than $\mathbb{R}$. The idea is to transform the observation $Y_i$ into the $\mathbb{R}$-valued random variable $g(\mu_i, Y_i)$, where $\mu_i$ is the forecast for $Y_i$ and $g$ is a proper scoring rule (see e.g. \cite{GRf07}). We then also transform the forecast into a random distribution function $\Gamma_g(\mu_i)$ as follows.

\begin{definition}
  Let $(S,\mathcal{S})$. Let $g(\mu,y)$ be a measurable function of a probability measure $\mu$ on $(S,\mathcal{S})$ and $y \in S$, taking values in $\mathbb{R}$. The function $\Gamma_g$ takes a probability measure $\mu$ on $(S,\mathcal{S})$ to the distribution function given by
\begin{equation*}
\Gamma_g(\mu) = \phi((y \mapsto g(\mu,y))_*(\mu)),
\end{equation*}
where $f_*(\mu)$ denotes the push-forward measure.
\end{definition}

$\Gamma_g$ is measurable by Lemma 3.2(ii) of \cite{Kll21}. We can then see that auto-calibration is preserved for the transformed sequences of forecasts and observations.

\begin{proposition}\label{transformation}
  Let $\mu_1$, \ldots, $\mu_n$ be random probability measures on a Borel space $(S,\mathcal{S})$, and let $Y_1$, \ldots, $Y_n$ be random elements in $(S,\mathcal{S})$, such that $(\mu_1, \ldots, \mu_n)$ is auto-calibrated as a rolling-event probability forecast sequence for observations $(Y_1, \ldots, Y_n)$. Let $g(\mu,y)$ be a measurable $\mathbb{R}$-valued function of a probability measure $\mu$ on $(S,\mathcal{S})$ and of $y \in S$.

  Then $(\Gamma_g(\mu_1), \ldots, \Gamma_g(\mu_n))$ is auto-calibrated as a rolling-event probability forecast sequence for observations $(g(\mu_1,Y_1), \ldots, g(\mu_n,Y_n))$.
\end{proposition}

We give a proof in the \hyperref[appn]{Appendix}.

Since the transformed observations are real-valued, and the transformed forecasts are random distribution functions, we can calculate the corresponding sequence of PIT values. We can then test whether the original forecast sequence was auto-calibrated by testing whether these PIT values are i.i.d. $\mathcal{U}([0,1])$.

\section{Synthetic PIT values}\label{syn_PITs}
In this section, we derive a testable consequence of auto-calibration for a fixed-event sequence of probability forecasts, initially requiring the observation to be real-valued and then allowing it to take values in any Borel space. Our approach is to use `synthetic observations' to obtain an auto-calibrated rolling-event forecast sequence, and then to use the Probability Integral Transform, which was reviewed in section \ref{rolling_event}. We call the resulting variables `synthetic PIT values'.

We shall need the following form of the powerful Disintegration Theorem---see Theorem 8.5(ii) of \cite{Kll21}.
\begin{theorem}\label{Disintegration}
  Let $\eta$ be a random element in a Borel space $(T,\mathcal{T})$ and let $\Psi \subset \mathcal{F}$ be a $\sigma$-algebra. Let $\mu$ be a version of $\mathcal{L}(\eta|\Psi)$ and let $H : \Omega \times T \to \mathbb{R}_{\geq 0}$ be $(\Psi \otimes \mathcal{T})/\mathcal{B}(\mathbb{R}_{\geq 0})$-measurable. Then
  \begin{equation*}
    \mathbb{E}[H(\eta)|\Psi] = \int H(t) \,\mu(dt).
  \end{equation*}
\end{theorem}

Our method will in fact allow us to test the martingale property for finite sequences of random probability measures; we will then use Proposition \ref{auto-calibration_equivalence} to apply it to auto-calibrated fixed-event probability forecast sequences as a special case. The following proposition allows us to rewrite the martingale property using some auxiliary random variables.
\begin{proposition}\label{synthetic_observations}
Let $\mu_1, \ldots, \mu_{n+1}$ be random probability measures on a Borel space $(S,\mathcal{S})$. Let $X_2$, \ldots, $X_{n+1}$ be random elements in $(S,\mathcal{S})$ such that
\begin{equation*}
  \mathcal{L}((X_2, \ldots, X_{n+1})|\mu_1, \ldots, \mu_{n+1}) = \mu_2 \otimes \ldots \otimes \mu_{n+1}.
\end{equation*}

  Then $(\mu_1, \ldots, \mu_{n+1})$ is a martingale if and only if $(\mu_1, \ldots, \mu_n)$ is auto-calibrated as a rolling-event probability forecast sequence for observations $(X_2, \ldots, X_{n+1})$.
\end{proposition}

\begin{proof}
Fix $i \in \{1, \ldots, n\}$. It suffices to show
\begin{eqnarray}
  \mathcal{L}(X_{i+1}|\mu_1, X_2, \ldots, \mu_{i-1}, X_i, \mu_i) & = & \mathbb{E}[\mu_{i+1}|\mu_1, \ldots, \mu_i, X_2, \ldots, X_i] \label{step1}\\
                                                                 & = & \mathbb{E}[\mu_{i+1}|\mu_1, \ldots, \mu_i]. \label{step2}
\end{eqnarray}

  Proof of equation (\ref{step1}):\, First, note that
  \begin{equation*}
    X_{i+1} \ci (X_2, \ldots, X_i) \, | \, \mu_1, \ldots, \mu_{i+1},
  \end{equation*}
  so by Theorem 8.9 of \cite{Kll21},
  \begin{equation*}
    \mathcal{L}(X_{i+1}|\mu_1, \ldots, \mu_{i+1}, X_2, \dots, X_i) = \mu_{i+1}.
  \end{equation*}
  
  Let $A \in \mathcal{S}$. Then by Theorem \ref{Disintegration}, 
  \begin{eqnarray*}
    \mathbb{E}[\mathbbm{1}(X_{i+1} \in A) | \mu_1, \ldots, \mu_{i+1}, X_2, \ldots, X_i] & = & \int \mathbbm{1}(x \in A) \, \mu_{i+1}(dx) \\
                                                                                        & = & \mu_{i+1}(A).
  \end{eqnarray*}

  Therefore, by the Tower Law,
  \begin{equation*}
    \mathbb{E}[\mathbbm{1}(X_{i+1} \in A)|\mu_1, \ldots, \mu_i, X_2, \ldots, X_i] = \mathbb{E}[\mu_{i+1}(A)|\mu_1, \ldots, \mu_i, X_2, \ldots, X_i],
  \end{equation*}
  which gives equation (\ref{step1}).

  Proof of equation (\ref{step2}):\, From the assumed joint distribution of $X_2$, \ldots, $X_{n+1}$, we have
  \begin{eqnarray*}
    \mathcal{L}((X_2, \ldots, X_i)|\mu_1, \ldots, \mu_{i+1}) & = & \mu_2 \otimes \ldots \otimes \mu_i \\
                                                             & = & \mathcal{L}((X_2, \ldots, X_i)|\mu_1, \ldots, \mu_i),
  \end{eqnarray*}
  and so by Theorem 8.9 of \cite{Kll21},
  \begin{equation*}
    \mu_{i+1} \ci (X_2, \ldots, X_i) \, | \, \mu_1, \ldots, \mu_i.
  \end{equation*}

  Then by Theorem 8.9 of \cite{Kll21} again,
  \begin{equation*}
    \mathcal{L}(\mu_{i+1}|\mu_1, \ldots, \mu_i, X_2, \ldots, X_i) = \mathcal{L}(\mu_{i+1}|\mu_1, \ldots, \mu_i).
  \end{equation*}
  
  Equation (\ref{step2}) now follows from Theorem \ref{Disintegration}.
\end{proof}

Proposition \ref{synthetic_observations} suggests that we can take a sequence of random probability measures and reinterpret it as a rolling-event sequence of forecasts for suitable variables $X_2$, \ldots, $X_{n+1}$, and shows that if the original sequence is a martingale then the rolling-event forecast sequence will be auto-calibrated. We call these variables $X_2$, \ldots, $X_{n+1}$ `synthetic observations', since they play the role of observations in this construction but are derived from the forecasts. Note if $\mu_{n+1} = \delta_{Y}$ for an observation $Y$ then we can choose $X_{n+1} = Y$.

The variable $X_i$ is a random draw from the probability distribution $\mu_i$, so to generate these synthetic observations we must be able to sample from probability distributions on $(S, \mathcal{S})$. In the case where $(S, \mathcal{S}) = (\mathbb{R}, \mathcal{B}(\mathbb{R}))$ we can write the forecasts as random distribution functions $F_1$, \ldots, $F_{n+1}$ and sample from them using their quantile functions.

For $F$ a distribution function, we let $F^{-1}$ be the corresponding quantile function, where for $t \in (0,1)$:
\begin{equation*}
  F^{-1}(t) = \inf \{y \in \mathbb{R}|F(y) \geq t\}.
\end{equation*}
The key property of $F^{-1}$ we shall need is that, for all $y \in \mathbb{R}$ and $t \in (0,1)$,
\begin{equation*}
  F^{-1}(t) \leq y \iff t \leq F(y).
\end{equation*}
The function $F^{-1}(t)$ is measurable in $t$ since it is non-decreasing. To see that $F^{-1}(t)$ is measurable in $F$ for fixed $t$, note that $F^{-1}(t) \in (-\infty, y] \iff F(y) \in [t, \infty)$.

The following Lemma then shows we can use $X_i = F_i^{-1}(U_i)$ as our synthetic observations in Proposition \ref{synthetic_observations}; results of this type are well known, but we include a proof in the \hyperref[appn]{Appendix} for convenience.
\begin{lemma}\label{F_inv_U}
  Let $F_1$, \ldots, $F_n$ be random distribution functions and let $\Psi \subset \mathcal{F}$. Let $U_1$, \ldots, $U_n$ be i.i.d. $\mathcal{U}([0,1])$ random variables independent of $\sigma(F_1, \ldots, F_n, \Psi)$. Then

\begin{equation*}
  \mathcal{L}((F_1^{-1}(U_1), \ldots, F_n^{-1}(U_n))|F_1, \ldots, F_n, \Psi) = \phi^{-1}(F_1) \otimes \ldots \otimes \phi^{-1}(F_n).
\end{equation*}
\end{lemma}

Since Proposition \ref{PIT_values} allows us to test the auto-calibration of rolling-event probability forecast sequences made for real-valued observations, we are now ready to test the auto-calibration of a fixed-event probability forecast sequence made for a real-valued observation.
\begin{proposition}\label{syn_PIT_values}
  Let $Y$ be an $\mathbb{R}$-valued random variable, and let $F_1, \ldots, F_n$ be probability forecasts for $Y$. Let $U_2$, \ldots, $U_n$, $V_1$, \ldots, $V_n$ be i.i.d. $\mathcal{U}([0,1])$ random variables independent of $\sigma(F_1, \ldots, F_n, Y)$.

  For each $i \in \{1, \ldots, n-1\}$, let $Z_i = Z(F_i, F_{i+1}^{-1}(U_{i+1}), V_i)$. Also, let $Z_n = Z(F_n, Y, V_n)$.

  Then the following are equivalent:
\begin{enumerate}[label=(\roman*)]
\item the sequence $(F_1, \ldots, F_n)$ is auto-calibrated;
\item for each $i \in \{1, \ldots, n\}$, $Z_i$ has distribution $\mathcal{U}([0,1])$ and is independent of $\sigma(F_1, \ldots, F_i)$.
\end{enumerate}

  Consequently, {\normalfont (i)} implies $Z_1$, \ldots, $Z_n$ are i.i.d. $\mathcal{U}([0,1])$.
\end{proposition}

\begin{proof} By Proposition \ref{auto-calibration_equivalence}, $(F_1, \ldots, F_n)$ is auto-calibrated if and only if $(F_1, \ldots, F_n, \delta_Y)$ is a martingale.

  By Lemma \ref{F_inv_U},
\begin{equation*}
\mathcal{L}((F_2^{-1}(U_2), \ldots, F_n^{-1}(U_n))|F_1, \ldots, F_n, \delta_Y) = \phi^{-1}(F_2) \otimes \ldots \otimes \phi^{-1}(F_n),
\end{equation*}
and so
\begin{equation*}
  \mathcal{L}((F_2^{-1}(U_2), \ldots, F_n^{-1}(U_n), Y)|F_1, \ldots, F_n, \delta_Y) = \phi^{-1}(F_2) \otimes \ldots \otimes \phi^{-1}(F_n) \otimes \delta_Y.
\end{equation*}

  Then by Proposition \ref{synthetic_observations}, $(F_1, \ldots, F_n, \delta_Y)$ is a martingale if and only if $(F_1, \ldots, F_n)$ is auto-calibrated as a rolling-event probability forecast sequence for observations 
\begin{equation*}
  (F_2^{-1}(U_2), \ldots, F_n^{-1}(U_n), Y).
\end{equation*}

  We are then done by Proposition \ref{PIT_values}.
\end{proof}

We call $Z_1$, \ldots, $Z_{n-1}$ as defined in Proposition \ref{syn_PIT_values} `synthetic PIT values', since they are PIT values based on synthetic observations; $Z_n$ is of course the ordinary PIT value for $F_n$. We can test that the sequence of synthetic PIT values is i.i.d. $\mathcal{U}([0,1])$ using any of the tests we would use for ordinary PIT values.

Finally for this section, we give one possible approach to testing the auto-calibration of a fixed-event probability forecast sequence for an observation that is not real-valued.

\begin{proposition}
Let $Y$ be a random element in a Borel space $(S, \mathcal{S})$, and let $(\mu_1, \ldots, \mu_n)$ be an auto-calibrated sequence of probability forecasts for $Y$.

  Let $g(\mu,y)$ be a measurable $\mathbb{R}$-valued function of a probability measure $\mu$ on $(S, \mathcal{S})$ and of $y \in S$. Let $X_2$, \ldots, $X_n$ be random elements in $(S, \mathcal{S})$ satisfying
\begin{equation*}
  \mathcal{L}((X_2, \ldots, X_n)|\mu_1, \ldots, \mu_n, Y) = \mu_2 \otimes \ldots \otimes \mu_n.
\end{equation*}

  Let $V_1$, \ldots, $V_n$ be i.i.d. $\mathcal{U}([0,1])$ random variables independent of $\sigma(\mu_1$, \ldots, $\mu_n$, $Y$, $X_2$, \ldots, $X_n)$.

  Then for each $i \in \{1, \ldots, n-1\}$, the variable $Z_i = Z(\Gamma_g(\mu_i), g(\mu_i, X_{i+1}), V_i)$ has distribution $\mathcal{U}([0,1])$ and is independent of $\sigma(\mu_1, \ldots, \mu_i)$. Also, $Z_n = Z(\Gamma_g(\mu_n), g(\mu_n,Y), V_n)$ has distribution $\mathcal{U}([0,1])$ and is independent of $\sigma(\mu_1, \ldots, \mu_n)$. Consequently, $Z_1$, \ldots, $Z_n$ are i.i.d. $\mathcal{U}([0,1])$.
\end{proposition}

\begin{proof}
  By Proposition \ref{auto-calibration_equivalence}, $(\mu_1, \ldots, \mu_n, \delta_Y)$ is a martingale.

  Note that
\begin{equation*}
\mathcal{L}((X_2, \ldots, X_n, Y)|\mu_1, \ldots, \mu_n, Y) = \mu_2 \otimes \ldots \otimes \mu_n \otimes \delta_Y.
\end{equation*}

  Then by Proposition \ref{synthetic_observations}, $(\mu_1, \ldots, \mu_n)$ is auto-calibrated as a rolling-event probability forecast sequence for observations $(X_2, \ldots, X_n, Y)$.

  Then by Proposition \ref{transformation}, $(\Gamma_g(\mu_1), \ldots, \Gamma_g(\mu_n))$ is auto-calibrated as a rolling-event probability sequence for observations $(g(\mu_1, X_2), \ldots, g(\mu_{n-1}, X_n), g(\mu_n, Y))$.

  We are then done by Proposition \ref{PIT_values}.
\end{proof}

\section{Probability forecast revisions}\label{revisions}
In this section we propose a definition of the revision between two probability forecasts made for a real-valued observation, and show that auto-calibration of a fixed-event probability forecast sequence can be rewritten in terms of the conditional expectations of the forecast revisions.

The following definition of the function $R$ was inspired in part by Czado, Gneiting and Held's equation (2) in \cite{Czd09}---note the resemblance of their equation to our equation (\ref{distribution_of_R_of_delta}) below.
\begin{definition}
  For $F, G$ distribution functions, $R(F, G) : (0,1) \mapsto [0,1]$ is given by
\begin{equation*}
  R(F, G)(t) = \int_\mathbb{R} \int_0^1 \mathbbm{1}(Z(F, y, v) \leq t) \,dv \,dG(y).
\end{equation*}
\end{definition}
By Proposition \ref{distribution_of_Z}, $R(F, G)(t) = Z(G, y_t, v_t)$, where $y_t$ and $v_t$ are defined as in the statement of Proposition \ref{distribution_of_Z}. Note if $F$ is continuous and strictly increasing, and $G$ is continuous, then $R(F, G)(t) = G(F^{-1}(t))$. For all distribution functions $F$, all $y \in \mathbb{R}$ and all $t \in (0,1)$,
\begin{equation}\label{distribution_of_R_of_delta}
  R(F, \phi(\delta_y))(t) = \begin{cases}
    0 & t < F_-(y) \\
    \frac{t-F_-(y)}{F(y)-F_-(y)} & F_-(y) \leq t < F(y) \\
    1 & F(y) \leq t,
  \end{cases}
\end{equation}
where if $F_-(y) = F(y)$ then the second case is not reached for any $t$.

If $F, G$ are distribution functions and $U, V$ are i.i.d. $\mathcal{U}([0,1])$ variables, then the CDF of $Z(F, G^{-1}(U), V)$ is given by, for all $t \in \mathbb{R}$,
\begin{equation*}
  \mathbb{P}(Z(F, G^{-1}(U), V) \leq t) = \begin{cases}
    0                           & t < 0 \\
    \lim_{s \to 0^+} R(F, G)(s) & t = 0 \\
    R(F, G)(t)                  & 0 < t < 1 \\
    1                           & 1 \leq t,
  \end{cases}
\end{equation*}
and so it is determined by $R(F, G)$; we can think of $R(F, G)$ as being the important part of the CDF of $Z(F, G^{-1}(U), V)$.

Let $I : (0,1) \to [0,1]$ be defined by $I(t) = t$. By Lemma \ref{PIT_is_uniform}, for $F, G$ distribution functions, $R(F, G) = I$ if and only if $F = G$.

Note $R(F, G)$ is non-decreasing and right-continuous. We use the $\sigma$-algebra on the set of non-decreasing and right-continuous functions from $(0,1)$ to $[0,1]$ generated by $(r \mapsto r(t))$ for $t \in (0,1)$. $R$ is a measurable function of $F$ and $G$ by Lemma 3.2(i) of \cite{Kll21}. For $R'$ a random non-decreasing and right-continuous function from $(0,1)$ to $[0,1]$ and $\Psi \subset \mathcal{F}$ we have the conditional expectation $\mathbb{E}[R'|\Psi]$ with the usual properties.

If $F$ and $G$ are forecasts for the same observation made by the same forecaster, with $G$ made after $F$, then we call $R(F, G)$ the revision from $F$ to $G$. This corresponds to Nordhaus's definition of the forecast revision $Q_2 - Q_1$ from point forecast $Q_1$ to point forecast $Q_2$ in \cite{nrd87}. The forward implication of the following proposition shows that if $(F, G)$ forms part of a martingale then the revision from $F$ to $G$ has expectation $I$.
\begin{proposition}\label{R}
  Let $F, G$ be random distribution functions, and let $\Psi \subset \mathcal{F}$ be a $\sigma$-algebra. Then the following are equivalent:

\begin{enumerate}[label=(\roman*)]
  \item $F = \mathbb{E}[G|\Psi]$;
  \item $F$ is $\Psi$-measurable and $\mathbb{E}[R(F,G)|\Psi] = I$.
\end{enumerate}
\end{proposition}

\begin{proof}
Let $U \sim \mathcal{U}([0,1])$ be independent of $\sigma(\Psi, G)$. Then for $y \in \mathbb{R}$,
\begin{eqnarray*}
\mathbb{P}(G^{-1}(U) \leq y|\Psi) & = & \mathbb{E}[\mathbb{P}(G^{-1}(U) \leq y|\Psi, G)|\Psi] \\
                                  & = & \mathbb{E}[G(y)|\Psi] \\
                                  & = & \mathbb{E}[G|\Psi](y),
\end{eqnarray*}
so $F_{G^{-1}(U)|\Psi} = \mathbb{E}[G|\Psi]$.

  Let $V \sim \mathcal{U}([0,1])$ be independent of $\sigma(\Psi, G, U)$. Then for $t \in (0,1)$,
\begin{eqnarray*}
  \mathbb{E}[R(F, G)|\Psi](t) & = & \mathbb{E}\left[ \int_\mathbb{R} \int_0^1 \mathbbm{1}(Z(F, G^{-1}(u), v) \leq t) \,dv \,du \middle| \Psi \right] \\
                              & = & \mathbb{E}\left[ \mathbbm{1}(Z(F, G^{-1}(U), V) \leq t) \middle| \Psi \right],
\end{eqnarray*}
  so $\mathbb{E}[R(F, G)|\Psi] = I$ if and only if $Z(F, G^{-1}(U), V)$ has distribution $\mathcal{U}([0,1])$ and is independent of $\Psi$.

  We are then done by Proposition \ref{auto-calibration_PIT}.
\end{proof}

This allows us to rewrite auto-calibration in terms of forecast revisions, as follows.
\begin{proposition}\label{auto-calibrated_revisions}
  Let $Y$ be an $\mathbb{R}$-valued random variable, and let $F_1$, \ldots, $F_n$ be probability forecasts for $Y$. Then the following are equivalent:
\begin{enumerate}[label=(\roman*)]
\item $(F_1, \ldots, F_n)$ is auto-calibrated;
\item for $i \in \{1, \ldots, n-1\}$, 
\begin{equation*}
\mathbb{E}[R(F_i, F_{i+1})|F_1, \ldots, F_i] = I,
\end{equation*}
and 
\begin{equation*}
\mathbb{E}[R(F_n, \phi(\delta_Y))|F_1, \ldots, F_n] = I.
\end{equation*}
\end{enumerate}
\end{proposition}

\begin{proof}
  By Proposition \ref{auto-calibration_equivalence}, $(F_1, \ldots, F_n)$ is auto-calibrated if and only if
\begin{equation*}
  (F_1, \ldots, F_n, \phi(\delta_Y))
\end{equation*}
is a martingale.

  By Proposition \ref{R}, for each $i \in \{1, \ldots, n-1\}$, $\mathbb{E}[F_{i+1}|F_1, \ldots, F_i] = F_i$ if and only if \begin{equation*}
\mathbb{E}[R(F_i, F_{i+1})|F_1, \ldots, F_i] = I.
\end{equation*}

  Also, $\mathbb{E}[\phi(\delta_Y)|F_1, \ldots, F_n] = F_n$ if and only if
\begin{equation*}
\mathbb{E}[R(F_n, \phi(\delta_Y))|F_1, \ldots, F_n] = I.
\end{equation*}
\end{proof}

Since $R(F, G)$ is the important part of the CDF of the synthetic PIT value between $F$ and $G$, a random draw from it will have the same distribution as a synthetic PIT value between $F$ and $G$. Proposition \ref{R} then suggests the following Lemma.

\begin{lemma}
  Let $F$, $G$ be random distribution functions and let $\Psi \subset \mathcal{F}$ be a $\sigma$-algebra. Let $U$ be a random variable with distribution $\mathcal{U}([0,1])$ and independent of $\sigma(F, G, \Psi)$.

  Then the following are equivalent:
\begin{enumerate}[label=(\roman*)]
  \item $\mathbb{E}[R(F, G)|\Psi] = I$;
  \item $R(F, G)^{-1}(U)$ has distribution $\mathcal{U}([0,1])$ and is independent of $\Psi$.
\end{enumerate}
\end{lemma}

\begin{proof}
For all $t \in (0,1)$,
\begin{eqnarray*}
  \mathbb{P}(R(F, G)^{-1}(U) \leq t|\Psi) & = & \mathbb{E}[\mathbb{P}(R(F, G)^{-1}(U) \leq t|F, G, \Psi)|\Psi] \\
                                          & = & \mathbb{E}[\mathbb{P}(U \leq R(F, G)(t)|F, G, \Psi)|\Psi] \\
                                          & = & \mathbb{E}[R(F, G)(t)|\Psi].
\end{eqnarray*}
\end{proof}

We can then rewrite Proposition \ref{auto-calibrated_revisions} to mirror Proposition \ref{syn_PIT_values}.

\begin{corollary}
  Let $Y$ be an $\mathbb{R}$-valued random variable, and let $F_1$, \ldots, $F_n$ be probability forecasts for $Y$. Let $U_1$, \ldots, $U_n$ be i.i.d. $\mathcal{U}([0,1])$ random variables independent of $\sigma(F_1, \ldots, F_n, Y)$.

  For each $i \in \{1, \ldots, n-1\}$, let $Z_i = R(F_i, F_{i+1})^{-1}(U_i)$. Let $Z_n = R(F_n, \phi(\delta_Y))^{-1}(U_n)$.

  Then the following are equivalent:
\begin{enumerate}[label=(\roman*)]
  \item $(F_1, \ldots, F_n)$ is auto-calibrated;
  \item for each $i \in \{1, \ldots, n\}$, $Z_i$ has distribution $\mathcal{U}([0,1])$ and is independent of $\sigma(F_1, \ldots, F_i)$.
\end{enumerate}

Consequently, {\normalfont (i)} implies $Z_1$, \ldots, $Z_n$ are i.i.d. $\mathcal{U}([0,1])$.
\end{corollary}

This allows us to test whether a fixed-event probability forecast sequence is auto-calibrated using only the sequence of forecast revisions. It may be possible instead to test the conditional structure of the forecast revisions described in Proposition \ref{auto-calibrated_revisions} directly, without using auxiliary variables like $U_1$, \ldots, $U_n$, but we leave this for future research.

\section{Examples}\label{examples}
We illustrate with three examples the behaviour of synthetic PIT values and forecast revisions when a fixed-event probability forecast sequence is not auto-calibrated.

\begin{example}
Let $A_2$, \ldots, $A_{n+1}$ be i.i.d. $N(0,1)$ random variables. Let
\begin{equation*}
  Y = \sum_{i=2}^{n+1} A_i.
\end{equation*}

Let $Q_1$, \ldots, $Q_n$ be defined by the following:
\begin{equation*}
  Q_i = \sum_{j=2}^{i} \left(1 - \frac{1}{2^{i+1-j}}\right) A_j,
\end{equation*}
so $Q_1 = 0$, and then consider a sequence of forecasts $F_1$, \ldots, $F_n$ defined by, for all $y \in \mathbb{R}$,
\begin{equation*}
  F_i(y) = \Phi\left(\frac{y - Q_i}{\sqrt{n+1-i}}\right),
\end{equation*}
where $\Phi$ is the standard Normal CDF.

Note
\begin{eqnarray*}
  Q_i & = & \frac{1}{2} \sum_{j=2}^{i-1} \left(1 - \frac{1}{2^{i-j}}\right) A_j + \frac{1}{2} \sum_{j=2}^i A_j \\
      & = & \frac{1}{2} Q_{i-1} + \frac{1}{2} \mathbb{E}[Y|A_2, \ldots, A_i],
\end{eqnarray*}
so we could imagine that when the forecaster needs to produce a probability forecast $F_i$, they are aware that the best forecast they could produce would be centered on $\mathbb{E}[Y|A_2, \ldots, A_i]$ but are reluctant to change too much from their previous forecast $F_{i-1}$, and so choose to compromise between the two. Nordhaus discusses the tendency for forecasts to be `smoothed' like this in \cite{nrd87}.

Let $U_2$, \ldots, $U_n$ be i.i.d. $\mathcal{U}([0,1])$ independent of $\sigma(A_2, \ldots, A_{n+1})$. Then for each $i \in \{1, \ldots, n-1\}$, $F_i$ is continuous, so the synthetic PIT value between $F_i$ and $F_{i+1}$ is
\begin{eqnarray*}
  F_i(F_{i+1}^{-1}(U_{i+1})) & = & F_i\left(Q_{i+1} + \sqrt{n-i}\,\Phi^{-1}(U_{i+1})\right) \\
                             & = & \Phi\left(\frac{Q_{i+1} - Q_i + \sqrt{n-i} \,\Phi^{-1}(U_{i+1})}{\sqrt{n+1-i}} \right) \\
                             & = & \Phi\left(\frac{\sum_{j=2}^i \frac{1}{2^{i+2-j}} A_j + \frac{1}{2} A_{i+1} + \sqrt{n-i} \,\Phi^{-1}(U_{i+1})}{\sqrt{n+1-i}}\right).
\end{eqnarray*}

Conditional on $A_2$, \ldots, $A_i$, we have
\begin{equation*}
  \frac{\sum_{j=2}^i \frac{1}{2^{i+2-j}} A_j + \frac{1}{2} A_{i+1} + \sqrt{n-i} \,\Phi^{-1}(U_{i+1})}{\sqrt{n+1-i}} \sim N\left(\frac{\sum_{j=2}^i \frac{1}{2^{i+2-j}} A_j}{\sqrt{n+1-i}}, \frac{4n - 4i + 1}{4n - 4i + 4}\right).
\end{equation*}

Using the identity
\begin{equation}\label{normal_identity}
  \int_\mathbb{R} \Phi(z) \,d\Phi\left(\frac{z - \mu}{\sigma}\right) = \Phi\left(\frac{\mu}{\sqrt{1+\sigma^2}}\right),
\end{equation}
we can see that
\begin{equation*}
  \mathbb{E}[F_i(F_{i+1}^{-1}(U_{i+1}))|A_2, \ldots, A_i] = \frac{1}{2} \qquad \text{only if} \qquad \sum_{j=2}^i \frac{1}{2^{i+2-j}} A_j = 0.
\end{equation*}

Note also $\sigma(F_1, \ldots, F_i) = \sigma(A_2, \ldots, A_i)$. Thus the conditional distribution of the synthetic PIT value between $F_i$ and $F_{i+1}$ given $F_1$, \ldots, $F_i$ is not always $\mathcal{U}([0,1])$.
\end{example}

\begin{example}
Let $A_2$, \ldots, $A_{n+1}$ be i.i.d. $N(0,1)$ random variables. Let
\begin{equation*}
  Y = \sum_{i=2}^{n+1} A_i.
\end{equation*}

Suppose a forecaster mistakenly believes $A_2$, \ldots, $A_{n+1}$ are i.i.d. $N(0,4)$, and so for $i \in \{1, \ldots, n\}$ they produce the forecast $F_i$ for $Y$ given by, for $y \in \mathbb{R}$,
\begin{equation*}
  F_i(y) = \Phi\left(\frac{y - A_2 - \ldots - A_i}{2 \sqrt{n+1-i}}\right),
\end{equation*}
where $\Phi$ is the standard Normal CDF.

Let $U_2$, \ldots, $U_n$ be i.i.d. $\mathcal{U}([0,1])$ independent of $A_2$, \ldots, $A_{n+1}$. Then for $i \in \{1, \ldots, n-1\}$, $F_i$ is continuous, so the synthetic PIT value between $F_i$ and $F_{i+1}$ is
\begin{eqnarray*}
  F_i(F_{i+1}^{-1}(U_{i+1})) & = & F_i\left(A_2 + \ldots + A_{i+1} + 2 \sqrt{n-i} \,\Phi^{-1}(U_{i+1})\right) \\
                             & = & \Phi\left(\frac{A_{i+1} + 2 \sqrt{n-i} \,\Phi^{-1}(U_{i+1})}{2 \sqrt{n+1-i}}\right).
\end{eqnarray*}

Since $A_{i+1}$ and $\Phi^{-1}(U_{i+1})$ are i.i.d. $N(0,1)$,
\begin{equation*}
  \frac{A_{i+1} + 2\sqrt{n-i} \,\Phi^{-1}(U_{i+1})}{2 \sqrt{n+1-i}} \sim N\left(0, \frac{4n - 4i + 1}{4n - 4i + 4}\right),
\end{equation*}
so $F_i(F_{i+1}^{-1}(U_{i+1})) \nsim \mathcal{U}([0,1])$. Note the synthetic PIT values are independent of each other in this case.

The forecast $F_i$ is a probability distribution with mean $A_2 + \ldots + A_i$, which is equal to $\mathbb{E}[Y|A_2, \ldots, A_i]$, so the sequence of forecast means is efficient as a point forecast sequence. This means that if we had followed existing practice by extracting the forecast means and testing whether their revisions were uncorrelated and had expectation $0$ (as described in \cite{nrd87}), then we would not have been able to detect the forecast miscalibration.

Since the forecast $F_i$ is continuous and strictly increasing, and the forecast $F_{i+1}$ is continuous, the revision $R(F_i, F_{i+1})$ is given by, for all $t \in \mathbb{R}$, $R(F_i, F_{i+1})(t) = F_{i+1}(F_i^{-1}(t))$. In fact,
\begin{equation*}
  R(F_i, F_{i+1})(t) = \Phi\left(\frac{2\sqrt{n+1-i}\, \Phi^{-1}(t) - A_{i+1}}{2\sqrt{n-i}}\right),
\end{equation*}
and so, using identity (\ref{normal_identity}) again,
\begin{equation*}
  \mathbb{E}[R(F_i, F_{i+1})(t)] = \Phi\left(\frac{2\sqrt{n+1-i}}{\sqrt{4n-4i+1}}\, \Phi^{-1}(t)\right).
\end{equation*}
Thus $\mathbb{E}[R(F_i, F_{i+1})] \neq I$, as expected.
\end{example}

\begin{example}
  We construct random distribution functions $F_1$, $F_2$, $F_3$ such that the synthetic PIT values $Z(F_1, F_2^{-1}(U_2), V_1)$ and $Z(F_2, F_3^{-1}(U_3), V_2)$ are i.i.d. $\mathcal{U}([0,1])$, but the synthetic PIT value $Z(F_1, F_3^{-1}(U_3), V_1)$ does not have distribution $\mathcal{U}([0,1])$. It may be surprising that this is possible, as there is no analogous possibility for rolling-event probability forecast sequences and ordinary PIT values. This example shows that there are fixed-event forecast sequences which are not auto-calibrated but whose miscalibration we cannot detect by calculating the synthetic PIT values, but also that for such sequences it may be possible to detect the miscalibration in a subsequence of the forecasts by calculating the synthetic PIT values for that subsequence.

Let $A$, $B$, $X$ and $Y$ be independent, with $A, B \sim \text{Bernoulli}\left(\frac{1}{2}\right)$ and $X, Y \sim \mathcal{U}([0,1])$.

Let $F_3$ be the random distribution function given by
\begin{equation*}
  F_3(y) = \begin{cases}
    0 & y < Y \\
    1 & y \geq Y.
  \end{cases}
\end{equation*}

Let $F_2$ be the random distribution function defined as follows. Note $\frac{5}{3}t - t^2 + \frac{1}{3}t^3$ is a continuous and strictly increasing function taking values from $0$ to $1$ on $[0,1]$. Let $G$ be its inverse function on $[0,1]$, which is also a continuous and strictly increasing function taking values from $0$ to $1$. If $B = 1$ and $X \in (0,1)$, let $F_2$ be given by:
\begin{equation*}
  F_2(y) = \begin{cases}
    0             & y < 0 \\
    \frac{3}{2X}y & 0 \leq y < \frac{2 X^2}{3} \\
    X             & \frac{2 X^2}{3} \leq y < X \\
    y             & X \leq y < 1 \\
    1             & 1 \leq y.
  \end{cases}
\end{equation*}
Otherwise, let $F_2$ be given by:
\begin{equation*}
  F_2(y) = \begin{cases}
    0    & y < 0 \\
    G(y) & 0 \leq y < 1 \\
    1    & 1 \leq y.
  \end{cases}
\end{equation*}

Since $F_2$ is continuous, the synthetic PIT value between $F_2$ and $F_3$ is $F_2(F_3^{-1}(U_3))$, where $U_3$ has distribution $\mathcal{U}([0,1])$ and is independent of $\sigma(A, B, X, Y)$. If $B = 1$ and $X \in (0,1)$,

\begin{equation*}
  \mathbb{P}(F_2(F_3^{-1}(U_3)) \leq t|B, X) = \begin{cases}
    0             & t < 0 \\
    \frac{2X}{3}t & 0 \leq t < X \\
    t             & X \leq t < 1 \\
    1             & 1 \leq t.
  \end{cases}
\end{equation*}

Otherwise,
\begin{equation*}
  \mathbb{P}(F_2(F_3^{-1}(U_3)) \leq t|B, X) = \begin{cases}
    0                                   & t < 0 \\
    \frac{5}{3}t - t^2 + \frac{1}{3}t^3 & 0 \leq t < 1 \\
    1                                   & 1 \leq t.
  \end{cases}
\end{equation*}

Then
\begin{eqnarray*}
  \mathbb{P}(F_2(F_3^{-1}(U_3)) \leq t) & = & \mathbb{E}[\mathbb{P}(F_2(F_3^{-1}(U_3)) \leq t|B, X)] \\
                                        & = & \begin{cases}
                                                0 & t < 0 \\
                                                t & 0 \leq t < 1 \\
                                                1 & 1 \leq t,
                                              \end{cases}
\end{eqnarray*}
so $F_2(F_3^{-1}(U_3))$ has distribution $\mathcal{U}([0,1])$.

Now let $F_1$ be the random distribution function defined as follows. If $A = 0$, $B = 1$ and $X \in (0,1)$, let $F_1$ be given by:
\begin{equation*}
  F_1(y) = \begin{cases}
    0                      & y < 0 \\
    \frac{1}{X}y           & 0 \leq y < \frac{2 X^2}{3} \\
    \frac{y+2X-2X^2}{3-2X} & \frac{2 X^2}{3} \leq y < X \\
    y                      & X \leq y < 1 \\
    1                      & 1 \leq y.
  \end{cases}
\end{equation*}
If $A = 1$, $B = 1$ and $X \in (0,1)$, let $F_1$ be given by:
\begin{equation*}
  F_1(y) = \begin{cases}
    0                           & y < 0 \\
    \frac{3}{X}y                & 0 \leq y < \frac{2X^2}{9} \\
    \frac{3}{4X}y + \frac{X}{2} & \frac{2X^2}{9} \leq y < \frac{2X^2}{3} \\
    X                           & \frac{2X^2}{3} \leq y < X \\
    y                           & X \leq y < 1 \\
    1                           & 1 \leq y.
  \end{cases}
\end{equation*}
Otherwise, let $F_1$ be given by:
\begin{equation*}
  F_1(y) = \begin{cases}
    0    & y < 0 \\
    G(y) & 0 \leq y < 1 \\
    1    & 1 \leq y.
  \end{cases}
\end{equation*}

Since $F_1$ is continuous, the synthetic PIT value between $F_1$ and $F_2$ is $F_1(F_2^{-1}(U_2))$, where $U_2$ has distribution $\mathcal{U}([0,1])$ and is independent of $\sigma(A, B, X, Y, U_3)$.

If $A = 0$, $B = 1$ and $X \in (0,1)$ then
\begin{equation*}
  \mathbb{P}(F_1(F_2^{-1}(U_2)) \leq t|A, B, X) = \begin{cases}
    0            & t < 0 \\
    \frac{3}{2}t & 0 \leq t < \frac{2X}{3} \\
    X            & \frac{2X}{3} \leq t < X \\
    t            & X \leq t < 1 \\
    1            & 1 \leq t.
  \end{cases}
\end{equation*}
If $A = 1$, $B = 1$ and $X \in (0,1)$ then
\begin{equation*}
  \mathbb{P}(F_1(F_2^{-1}(U_2)) \leq t|A, B, X) = \begin{cases}
    0            & t < 0 \\
    \frac{1}{2}t & 0 \leq t < \frac{2X}{3} \\
    2t - X       & \frac{2X}{3} \leq t < X \\
    t            & X \leq t < 1 \\
    1            & 1 \leq t.
  \end{cases}
\end{equation*}
Otherwise,
\begin{equation*}
  \mathbb{P}(F_1(F_2^{-1}(U_2)) \leq t|A, B, X) = \begin{cases}
    0 & t < 0 \\
    t & 0 \leq t < 1 \\
    1 & 1 \leq t.
  \end{cases}
\end{equation*}

Therefore,
\begin{eqnarray*}
  \mathbb{P}(F_1(F_2^{-1}(U_2)) \leq t|B, X) & = & \mathbb{E}[\mathbb{P}(F_1(F_2^{-1}(U_2)) \leq t|A, B, X)|B, X] \\
                                             & = & \begin{cases}
                                               0 & t < 0 \\
                                               t & 0 \leq t < 1 \\
                                               1 & 1 \leq t,
                                             \end{cases}
\end{eqnarray*}
so $F_1(F_2^{-1}(U_2))$ has distribution $\mathcal{U}([0,1])$ and is independent of $\sigma(B,X)$. Since it is clearly also independent of $\sigma(Y,U_3)$, it is independent of $F_2(F_3^{-1}(U_3))$.

Finally, the synthetic PIT value between $F_1$ and $F_3$ is $F_1(F_3^{-1}(U_3))$. If $A = 0$, $B = 1$ and $X \in (0,1)$ then
\begin{equation*}
  \mathbb{P}(F_1(F_3^{-1}(U_3)) \leq t|A, B, X) = \begin{cases}
    0                    & t < 0 \\
    Xt                   & 0 \leq t < \frac{2X}{3} \\
    3t - 2Xt - 2X + 2X^2 & \frac{2X}{3} \leq t < X \\
    t                    & X \leq t < 1 \\
    1                    & 1 \leq t.
  \end{cases}
\end{equation*}
If $A = 1$, $B = 1$ and $X \in (0,1)$ then
\begin{equation*}
  \mathbb{P}(F_1(F_3^{-1}(U_3)) \leq t|A, B, X) = \begin{cases}
    0                              & t < 0 \\
    \frac{X}{3}t                   & 0 \leq t < \frac{2X}{3} \\
    \frac{4X}{3}t - \frac{2X^2}{3} & \frac{2X}{3} \leq t < X \\
    t                              & X \leq t < 1 \\
    1                              & 1 \leq t.
  \end{cases}
\end{equation*}
Otherwise,
\begin{equation*}
  \mathbb{P}(F_1(F_3^{-1}(U_3)) \leq t|A, B, X) = \begin{cases}
    0                                   & t < 0 \\
    \frac{5}{3}t - t^2 + \frac{1}{3}t^3 & 0 \leq t < 1 \\
    1                                   & 1 \leq t.
  \end{cases}
\end{equation*}

Then for $t \in \left(\frac{2}{3},1\right)$,
\begin{eqnarray*}
  \mathbb{P}(F_1(F_3^{-1}(U_3)) \leq t) & = & \mathbb{E}[\mathbb{P}(F_1(F_3^{-1}(U_3)) \leq t|A, B, X)] \\
                                        & = & -\frac{5}{36} + \frac{3}{2}t - \frac{1}{2}t^2 + \frac{5}{36}t^3 \\
                                        & \neq & t,
\end{eqnarray*}
so $F_1(F_3^{-1}(U_3)) \nsim \mathcal{U}([0,1])$.
\end{example}

\section{Discussion}\label{discussion}
We have presented a method for testing whether a fixed-event probability forecast sequence is calibrated. In many ways our definition of auto-calibration and the consequences of it are analogous to Nordhaus' definition of weak efficiency for a fixed-event point forecast sequence in \cite{nrd87}, and the consequences of that. A key difference is that we have presented a fixed-size statistical test for auto-calibration which requires no further assumptions on the forecasts, but as far as we are aware such a test does not exist for weak efficiency. This difference seems reasonable: in producing probability forecasts rather than point forecasts, the forecaster has provided us with much more information about their beliefs.

There are other issues to consider before applying our method to test real-world forecast sequences: in particular, the forecasts may not initially be presented as probability measures, and so we may need to decide how best to interpret them as such, and we will need to choose a method for testing whether the synthetic PIT values are i.i.d. $\mathcal{U}([0,1])$. We do not address these issues here.

Proposition \ref{syn_PIT_values} in fact shows that under auto-calibration each synthetic PIT value is independent not only of the other synthetic PIT values, but also of the forecasts earlier in the sequence. Similarly, Proposition \ref{PIT_values} shows that for rolling-event probability forecast sequences, auto-calibration implies each PIT value is independent not only of the other PIT values but also of the earlier forecasts and observations. It may be preferable to test the stronger independence properties of synthetic PIT values, perhaps using regression trees as described in Chapter 3 of \cite{Mds23}. Also, the null hypothesis could be modified to require each forecast to be conditionally auto-calibrated given a larger information set representing some of the data the forecaster had access to when their forecast was made, which we would expect them to have incorporated into their forecast. We could then also test whether each synthetic PIT value is independent of the corresponding one of these information sets.

In a test of auto-calibration using synthetic PIT values, the null hypothesis of auto-calibration requires each forecast to be related to the distribution of the observation. However, the test should be thought of as primarily assessing the internal consistency of the forecast sequence, meaning whether each adjacent pair of forecasts in the sequence shows the forecaster updating their prediction coherently, rather than whether the forecasts are all aligned with the observation. To see this, consider an auto-calibrated fixed-event sequence of forecasts made for one observation, and imagine testing whether it is auto-calibrated as a fixed-event probability forecast sequence for a different, completely unrelated observation. Only the last synthetic PIT value in the sequence---or rather, the ordinary PIT value for the last forecast in the sequence---is affected by the modification to the observation, so the synthetic PIT values earlier in the sequence are still i.i.d. $\mathcal{U}([0,1])$. Any method for testing whether the full sequence of synthetic PIT values is i.i.d. $\mathcal{U}([0,1])$ is then unlikely to lead to a rejection of the null hypothesis, even though presumably none of the forecasts are close to being calibrated for the new observation.

The consequences of auto-calibration we have developed for a fixed-event probability forecast sequence are really consequences of the martingale property, and we could use them to test whether a sequence of random probability measures is a martingale with synthetic PIT values. In particular, we could state a proposition similar to Proposition \ref{syn_PIT_values} but giving an equivalence with the martingale property rather than with auto-calibration, and its proof would be essentially the same. Given a fixed-event probability forecast sequence, we could ignore the observation entirely and test whether the forecast sequence is a martingale. If we rejected that null hypothesis we would then have found that the forecast sequence is not auto-calibrated for any observation.

\begin{appendix}
\section*{Additional proofs}\label{appn}

\begin{proof}[Proof of Proposition \ref{auto-calibration_equivalence}]
  $\text{(i)} \Rightarrow \text{(ii)}$:\, For $i \in \{1, \ldots, n\}$, $\mu_i$ is $\Psi_i$-measurable. Then, since the $\Psi$s form a filtration, for $i \in \{1,\ldots,n\}$, $\sigma(\mu_1,...\mu_i) \subset \Psi_i$. For all $A \in \mathcal{S}$, by the Tower Law:
\begin{eqnarray*}
  \mathbb{E}[\mathbbm{1}(Y \in A)|\mu_1, \ldots, \mu_i] & = & \mathbb{E}[\mathbb{E}[\mathbbm{1}(Y \in A)|\Psi_i]|\mu_1, \ldots, \mu_i] \\
                                                        & = & \mathbb{E}[\mu_i(A)|\mu_1, \ldots, \mu_i] \\
                                                        & = & \mu_i(A).
\end{eqnarray*}
  $\text{(ii)} \Rightarrow \text{(i)}$:\, For each $i \in \{1, \ldots, n\}$, let $\Psi_i = \sigma(\mu_1, \ldots, \mu_i)$. \\
  $\text{(ii)} \Rightarrow \text{(iii)}$:\, For each $i \in \{1, \ldots, n-1\}$ and for all $A \in \mathcal{S}$,
\begin{eqnarray*}
  \mu_i(A) & = & \mathbb{E}[\mathbbm{1}(Y \in A)|\mu_1, \ldots, \mu_i] \\
           & = & \mathbb{E}[\mathbb{E}[\mathbbm{1}(Y \in A)|\mu_1, \ldots, \mu_{i+1}]|\mu_1, \ldots, \mu_i] \\
           & = & \mathbb{E}[\mu_{i+1}(A)|\mu_1, \ldots, \mu_i].
\end{eqnarray*}
  $\text{(iii)} \Rightarrow \text{(ii)}$:\, We show $\mu_i = \mathcal{L}(Y|\mu_1, \ldots, \mu_i)$ by induction from $i = n$ to $i = 1$. The base case is trivial. For $i \in \{1, \ldots, n-1\}$, assume $\mu_{i+1} = \mathcal{L}(Y|\mu_1, \ldots, \mu_{i+1})$. Then for all $A \in \mathcal{S}$,
\begin{eqnarray*}
  \mu_i(A) & = & \mathbb{E}[\mu_{i+1}(A)|\mu_1, \ldots, \mu_i] \\
           & = & \mathbb{E}[\mathbb{E}[\mathbbm{1}(Y \in A)|\mu_1, \ldots, \mu_{i+1}]|\mu_1, \ldots, \mu_i] \\
           & = & \mathbb{E}[\mathbbm{1}(Y \in A)|\mu_1, \ldots, \mu_i],
\end{eqnarray*}
  so indeed $\mu_i = \mathcal{L}(Y|\mu_1, \ldots, \mu_i)$. \\
  $\text{(iii)} \Leftrightarrow \text{(iv)}$:\, The sequence $(\mu_1, \ldots, \mu_n, \delta_Y)$ is a martingale if and only if $(\mu_1, \ldots, \mu_n)$ is a martingale and for all $A \in \mathcal{S}$,
\begin{eqnarray*}
  \mu_n(A) & = & \mathbb{E}[\delta_Y(A)|\mu_1,\ldots,\mu_n] \\
           & = & \mathbb{E}[\mathbbm{1}(Y \in A)|\mu_1,\ldots,\mu_n],
\end{eqnarray*}
which says exactly that $\mu_n = \mathcal{L}(Y|\mu_1,\ldots,\mu_n)$.
\end{proof}

\begin{proof}[Proof of Proposition \ref{distribution_of_Z}]
  For $y < y_t$, $F(y) \leq t$ since $F$ is non-decreasing, so
\begin{equation}\label{F_minus_y_t}
  F_-(y_t) \leq t.
\end{equation}

  For $y > y_t$, $F(y) > t$. Since $F(y_t) = \lim_{y \to y_t^+} F(y)$,
\begin{equation}\label{F_y_t}
  F(y_t) \geq t.
\end{equation}

  If $F_-(y_t) < F(y_t)$ then $0 \leq v_t \leq 1$ follows from inequalities (\ref{F_minus_y_t}) and (\ref{F_y_t}), and $Z(F, y_t, v_t) = t$ is straightforward.

  If $F_-(y_t) = F(y_t)$ then in fact $F_-(y_t) = F(y_t) = t$ by inequalities (\ref{F_minus_y_t}) and (\ref{F_y_t}), which gives $Z(F, y_t, v_t) = (0)(t) + (1)(t) = t$.

  For $y < y_t$ and $v \in [0,1]$, $Z(F, y, v) \leq t$ since $F(y) \leq t$. For $y > y_t$ and $v \in [0,1]$, $Z(F, y, v) > t$ since $F(y) > t$. For $v \in [0,v_t]$, $Z(F, y_t, v) \leq Z(F, y_t, v_t) = t$. And finally, for $v \in (v_t,1]$, we have $F_-(y_t) < F(y_t)$ and so $Z(F, y_t, v) > Z(F, y_t, v_t) = t$. Thus, for $y \in \mathbb{R}$ and $v \in [0,1]$, $Z(F, y, v) \leq t$ if and only if either $y < y_t$, or $y = y_t$ and $v \leq v_t$.

  Then
\begin{eqnarray*}
  \int_\mathbb{R} \int_0^1 \mathbbm{1}(Z(F, y, v) \leq t) \,dv \,dG(y) & = & \int_{(-\infty, y_t)} \int_0^1 \,dv \,dG(y) + \int_{\{y_t\}} \int_0^{v_t} \,dv \,dG(y) \\
                                                                       & = & G_-(y_t) + v_t \left(G(y_t) - G_-(y_t)\right) \\
                                                                       & = & Z(G, y_t, v_t),
\end{eqnarray*}
so we are done.
\end{proof}

  \begin{proof}[Proof of Proposition \ref{auto-calibration_PIT}] $\text{(i)} \Rightarrow \text{(ii)}$:\, Let $t \in (0,1)$. We shall show that
\begin{equation*}
  \mathbb{E}[\mathbbm{1}(Z(F,Y,V) \leq t)|\Psi] = t,
\end{equation*}
  so that the conditional distribution of $Z(F,Y,V)$ given $\Psi$ is $\mathcal{U}([0,1])$, as required.

  We apply Theorem \ref{Disintegration} with variable $\xi = (Y,V)$, $\sigma$-algebra $\Psi$ and random function $H(y,v) = \mathbbm{1}(Z(F,y,v) \leq t)$. Let $\mu = \phi^{-1}(F)$, so that $\mu$ is a version of $\mathcal{L}(Y|\Psi)$, and let $\lambda$ be the Lebesgue measure on $[0,1]$. Note that $\mathcal{L}((Y,V)|\Psi) = \mu \otimes \lambda$. We then have
\begin{eqnarray*}
  \mathbb{E}[\mathbbm{1}(Z(F,Y,V) \leq t)|\Psi] & = & \int \mathbbm{1}(Z(F,y,v) \leq t) \,(\mu \otimes \lambda)(dy,dv) \\
                                                & = & \int_\mathbb{R} \int_0^1 \mathbbm{1}(Z(F,y,v) \leq t) \,dv \,dF(y),
\end{eqnarray*}
  so we are done by Lemma \ref{PIT_is_uniform}.

$\text{(ii)} \Rightarrow \text{(i)}$:\, Let $G$ be a version of $F_{Y|\Psi}$.

  Since $F$ is $\Psi$-measurable, by Theorem \ref{Disintegration} we have for $t \in (0,1)$,
\begin{equation*}
  \mathbb{E}[\mathbbm{1}(Z(F, Y, V) \leq t)|\Psi] = \int_\mathbb{R} \int_0^1 \mathbbm{1}(Z(F, y, v) \leq t) \,dv \,dG(y).
\end{equation*}

  But $Z(F, Y, V)$ has conditional distribution $\mathcal{U}([0,1])$ given $\Psi$, so for $t \in (0,1)$,
\begin{equation*}
\mathbb{E}[\mathbbm{1}(Z(F, Y, V) \leq t)|\Psi] = t.
\end{equation*}

  Thus for $t \in (0,1)$,
\begin{equation}\label{a}
  \int_\mathbb{R} \int_0^1 \mathbbm{1}(Z(F, y, v) \leq t) \,dv \,dG(y) = t
\end{equation}
almost surely.

  Then with probability $1$, equation (\ref{a}) holds for $t \in (0,1) \cap \mathbb{Q}$. But both sides are right-continuous in $t$, so in fact with probability $1$, equation (\ref{a}) holds for $t \in (0,1)$.

  Then by Lemma \ref{PIT_is_uniform}, $F = G$ almost surely, so $F$ is a version of $F_{Y|\Psi}$ as required.
\end{proof}

\begin{proof}[Proof of Proposition \ref{transformation}]
  Let $i \in \{1, \ldots, n\}$. We have
\begin{equation*}
  \mathcal{L}(Y_i|\mu_1, Y_1, \ldots, \mu_{i-1}, Y_{i-1}, \mu_i) = \mu_i.
\end{equation*}

  Then for $A \in \mathcal{B}(\mathbb{R})$, using Theorem \ref{Disintegration},
\begin{eqnarray*}
\mathbb{E}[\mathbbm{1}(g(\mu_i,Y_i) \in A)|\mu_1, Y_1, \ldots, \mu_{i-1}, Y_{i-1}, \mu_i] & = & \int \mathbbm{1}(g(\mu_i,y) \in A) \,\mu_i(dy) \\
    & = & \int \mathbbm{1}(z \in A) \,((y \mapsto g(\mu_i,y))_*(\mu_i))(dz) \\
    & = & ((y \mapsto g(\mu_i,y))_*(\mu_i))(A) \\
    & = & \phi^{-1}(\Gamma_g(\mu_i))(A).
\end{eqnarray*}

Therefore
\begin{equation*}
  \mathcal{L}(g(\mu_i,Y_i)|\Gamma_g(\mu_1), g(\mu_1,Y_1), \ldots, \Gamma_g(\mu_{i-1}), g(\mu_{i-1},Y_{i-1}), \Gamma_g(\mu_i)) = \phi^{-1}(\Gamma_g(\mu_i)),
\end{equation*}
so we are done.
\end{proof}

\begin{proof}[Proof of Lemma \ref{F_inv_U}]
  For $i \in \{1, \ldots, n\}$, let $\mu_i = \phi^{-1}(F_i)$.

  To show that
\begin{equation*}
  \mathcal{L}((F_1^{-1}(U_1), \ldots, F_n^{-1}(U_n))|\mu_1, \ldots, \mu_n, \Psi) = \mu_1 \otimes \ldots \otimes \mu_n,
\end{equation*}
it suffices to show that $\mu_1 \otimes \ldots \otimes \mu_n$ satisfies the defining property of the regular conditional distribution on all sets of the form
\begin{equation*}
  (-\infty,y_1] \times \ldots \times (-\infty,y_n]
\end{equation*}
  for $y_1, \ldots, y_n \in \mathbb{R}$, by a monotone-class argument.

  Using Theorem \ref{Disintegration}, for $y_1, \ldots, y_n \in \mathbb{R}$,
\begin{eqnarray*}
\mathbb{E}\left[ \prod_{i=1}^n \mathbbm{1}(F_i^{-1}(U_i) \leq y_i) \middle| \mu_1, \ldots, \mu_n, \Psi \right] & = & \mathbb{E}\left[ \prod_{i=1}^n \mathbbm{1}(U_i \leq F_i(y_i)) \middle| \mu_1, \ldots, \mu_n, \Psi \right] \\
    & = & \int_0^1 \cdots \int_0^1 \left( \prod_{i=1}^n \mathbbm{1}(u_i \leq F_i(y_i)) \right) \,d{u_1} \cdots \,d{u_n} \\
    & = & \prod_{i=1}^n \left( \int_0^1 \mathbbm{1}(u_i \leq F_i(y_i)) \,d{u_i} \right) \\
    & = & \prod_{i=1}^n F_i(y_i),
\end{eqnarray*}
so we are done.
\end{proof}
\end{appendix}

\begin{acks}[Acknowledgments]
  Thomas Wilkinson was supported by the Engineering and Physical Sciences Research Council [grant number EP/W524451/1].
\end{acks}

\bibliographystyle{imsart-number}
\bibliography{refs}

\end{document}